\documentclass[a4paper,12pt]{article}

\usepackage{t1enc}
\usepackage[latin1]{inputenc}
\usepackage[english]{babel}
\usepackage{amsmath,amsthm}
\usepackage{amsfonts}
\usepackage{latexsym}
\usepackage{graphicx}
\usepackage{txfonts}
\usepackage[natural]{xcolor}
\usepackage{lineno}
\usepackage{pifont}
\usepackage{mathrsfs}
\usepackage{fancyhdr}
\usepackage{indentfirst}

\newtheorem{theorem}{Theorem}[section]
\newtheorem{corollary}{Corollary}[section]

\title{{Inclusion-exclusion by ordering-free cancellation}}
\author{Yin Chen, Jianguo Qian\thanks{Corresponding author: jgqian@xmu.edu.cn.}\\
{\footnotesize School of Mathematical Sciences, Xiamen University, Xiamen 361005, P. R. China}}

\date{}
\begin{document}
 \maketitle
\begin{abstract}
Whitney's broken circuit theorem gives a graphical example to reduce the number of the terms in the sum of the inclusion-exclusion formula by a  predicted cancellation.  So far, the known cancellations for the formula strongly depend on the prescribed (linear or partial) ordering on the index set. We give a new cancellation method, which does not require any ordering on the index set. Our method extends all the `ordering-based' methods known in the literatures and in general reduces more terms. As examples, we use our method to improve some relevant results on graph polynomials.  \\

\end{abstract}
\section{Introduction}
\label{intro}

Let $(\Omega,\mathscr{A},\mu)$ be a measure space, $P$ be a finite index set and $\{A_p\}_{p\in P}\subseteq \mathscr{A}$ be a family of measurable sets.  The formula
\begin{equation}
\mu\left(\bigcap_{p\in P}{\overline{A}_p}\right)=\sum_{I\subseteq P}(-1)^{|I|}\mu\left(\bigcap_{i\in I}A_i\right)
\end{equation}
is known as the principle of inclusion-exclusion, where $\overline{A}_p$ denotes the complement of $A_p$.

The principle of inclusion-exclusion is a classic counting technique in combinatorics and has been extensively studied \cite{Dohmen02,Dohmen01,Narushima01,Riordan01,Rota01,Whitney01}. Since the sum on the right side of Eq.(1) ranges over a large number of terms, it is natural to ask whether fewer terms would give the same result, that is, is it possible to reduce the number of terms by predicted cancellation? Lots of the answers to this question have been given by several authors. A well-known example is the one given by Whitney \cite{Whitney01} in 1932 for chromatic polynomial of a graph, which states that the calculation of  a chromatic polynomial can be restricted to the collection of those sets of edges which do not include any broken circuit as a subset.

Various cancellations for the inclusion-exclusion  principle were given from the perspective of both combinatorics and graph theory in the literatures. In \cite{Narushima01}, Narushima presented a cancellation for the inclusion-exclusion  principle, depending on a prescribed ordering on the index set $P$. This result was later improved by Dohmen \cite{Dohmen02}. Using the same technique,  Dohmen \cite{Dohmen01} also established an abstraction of Whitney's broken circuit theorem, which  not only applies to the chromatic polynomial, but also to other graph polynomials, see \cite{Dohmen05,Dohmen08,Dohmen01,Liao01,Trinks} for details.

 So far, the known cancellation methods for  inclusion-exclusion principle strongly depend on the prescribed (linear or partial) ordering on the index set $P$. In this article we establish a new cancellation method, which  does not require any ordering on $P$. Our method extends all the `ordering-based' methods given in the previous literatures and in general may reduce more terms. As examples,  we use our `ordering-free' method to improve the relevant results on the chromatic polynomial of hypergraphs, the independence polynomial and domination polynomial of graphs.

\section{Inclusion-exclusion by predicted cancellations}

For a subset $B$ of a poset (partially ordered set) $P$, let $B'$ denote
the set of upper bounds of $B$ which are not in $B$, that is,
\begin{equation*}
B'=\{p\in P:p>b\ \ {\rm for\ any}\ \ b\in B\}.
\end{equation*}

In \cite{Narushima01}, Narushima presented a cancellation for the inclusion-exclusion  principle on semilattices. This result was later extended to many forms. The  following one was given by Dohmen \cite{Dohmen02}:

\begin{theorem}\label{thm1} \cite{Dohmen02} Let $(\Omega,\mathscr{A},\mu)$ be a measure space, $P$  be a poset and $\{A_p\}_{p\in P}\subseteq \mathscr{A}$ be a family of measurable sets. If $\mathfrak{X}$ is a class of subsets of $P$ such that
\begin{equation}
\bigcap_{p\in B}A_p\subseteq \bigcup_{p\in B'}A_{p}
\end{equation}
for each $B\in\mathfrak{X}$. Then
\begin{equation}\label{main}
\mu\left(\bigcap_{p\in P}{\overline{A}_p}\right)=\sum_{I\in 2^{P}\setminus\mathfrak{I}}(-1)^{|I|}\mu\left(\bigcap_{i\in I}A_i\right),
\end{equation}
where $2^P$ is the power set of $P$ and $\mathfrak{I}=\{I\subseteq P:I\supseteq B\ \ {\rm for\ some}\ \ B\in\mathfrak{X}\}.$
\end{theorem}

Let $\{B_1,B^*_1\},\{B_2,B^*_2\},\cdots,\{B_k,B^*_k\}$ be pairs of subsets of $P$ with $B_i\cap B^*_i=\emptyset$ for every $i\in\{1,2,\cdots,k\}$. Denote
$$\mathscr{B}_i=\{I\subseteq P: I\supseteq B_i,  I\nsupseteq B_j\setminus B^*_i\  {\rm for}\ j<i\}$$
and
 \begin{equation}
 \mathscr{B}=\mathscr{B}_1\cup\mathscr{B}_2\cup\cdots\cup\mathscr{B}_k.
 \end{equation}
We note that $\mathscr{B}_i$ is empty when $B_j\setminus B^*_i\subseteq B_i$ for some $j<i$ since there is no $I$ satisfies the requirement.

We now give our main result which   does not require any ordering on $P$.

\begin{theorem}\label{thm2} Let $(\Omega,\mathscr{A},\mu)$ be a measure space, $P$ be a set and $\{A_p\}_{p\in P}\subseteq \mathscr{A}$ be a family of measurable sets. Let $\{B_1,B^*_1\},\{B_2,B^*_2\},\cdots,\{B_k,B^*_k\}$ be pairs of subsets of $P$. If $B_i\cap B^*_i=\emptyset$ and
\begin{equation}
\bigcap_{p\in B_i}A_p\subseteq \bigcup_{p\in  B^*_i}A_p
\end{equation}
for every $i\in\{1,2,\cdots,k\}$, then
\begin{equation}\label{main}
\mu\left(\bigcap_{p\in P}{\overline{A}_p}\right)=\sum_{I\in 2^P\setminus \mathscr{B}}(-1)^{|I|}\mu\left(\bigcap_{i\in I}A_i\right).
\end{equation}
\end{theorem}
\begin{proof} Let $I\in\mathscr{B}$. Then $I\in\mathscr{B}_i$ for some $i\in\{1,2,\cdots,k\}$.  We claim that such $\mathscr{B}_i$ is unique. In fact, suppose to the contrary that  $I\in\mathscr{B}_j$ and, with no loss of generality, that $j<i$. Then by the definition of $\mathscr{B}_i$, $I\nsupseteq B_j$. This contradicts that $I\in \mathscr{B}_j$. As a result,
$$\mathscr{B}_1,\mathscr{B}_2,\cdots,\mathscr{B}_k$$
are pairwise disjoint and therefore,  (4) is a partition of $\mathscr{B}$.

For $I\in \mathscr{B}_i$, let $I^*=I\setminus B^*_i$. Since $I\supseteq B_i$ and $B_i\cap B^*_i=\emptyset$, we have $I^*\supseteq B_i$. We claim that $I^*\cup D^*_i\in\mathscr{B}_i$ for any $D^*_i\subseteq  B^*_i$.

Suppose to the contrary that $I^*\cup D^*_i\notin \mathscr{B}_i$ for some $D^*_i\subseteq B^*_i$.  Since $I^*\cup D^*_i\supseteq I^*\supseteq B_i$, so by the definition of $\mathscr{B}_i$, $I^*\cup D^*_i\supseteq B_j\setminus B^*_i$ for some $j<i$. Thus, $I^*\supseteq B_j\setminus B^*_i$ since $D^*_i\subseteq B^*_i$. Therefore,
$$I\supseteq I^*\supseteq B_j\setminus B^*_i.$$

This is a contradiction because $I\in \mathscr{B}_i$, i.e., $I\nsupseteq B_j\setminus B^*_i$. Our claim follows.

For $I\in \mathscr{B}_i$, let
 $$\langle I\rangle=\{I^*\cup D^{*}_i:D^*_i\subseteq B^*_i\}.$$
Then
$$\sum_{J\in \langle I\rangle}(-1)^{|J|}\mu\left(\bigcap_{p\in J}A_p\right)=\sum_{D^*_i\subseteq B^*_i}(-1)^{|I^*\cup D^*_i|}\mu\left(\bigcap_{p\in I^*}A_p\cap\bigcap_{p\in D^*_i}A_p\right)$$
$$=(-1)^{|I^*|}\sum_{D^*_i\subseteq B^*_i}(-1)^{|D^*_i|}\mu\left(\bigcap_{p\in I^*}A_p\cap\bigcap_{p\in D^*_i}A_p\right)$$
$$=(-1)^{|I^*|}\mu\left(\bigcap_{p\in I^*}A_p\cap\bigcap_{p\in B^*_i}\overline{A}_p\right),$$
where the last equality holds by the principle of inclusion-exclusion. Notice that $\bigcap_{p\in B^*_i}\overline{A}_p$ is the complement of $\bigcup_{p\in B^*_i}A_p$. So by (5),
$$\bigcap_{p\in I^*}A_p\cap\bigcap_{p\in B^*_i}\overline{A}_p=\emptyset$$
 since $I^*\supseteq B_i$. Therefore,
\begin{equation}
\sum_{J\in \langle I\rangle}(-1)^{|J|}\mu\left(\bigcap_{p\in J}A_p\right)=0.
\end{equation}

Finally,  for any $I,J\in  \mathscr{B}_i$, by the definition of $I^*$ we can see that either  $\langle J\rangle\cap \langle I\rangle=\emptyset$ or $\langle J\rangle=\langle I\rangle$. In other words,
$\bigcup_{I\in\mathscr{B}_i}\langle I\rangle$ is a partition of $\mathscr{B}_i$, written by
$$\mathscr{B}_i=\langle I_1\rangle\cup\langle I_2\rangle\cup\cdots\cup\langle I_t\rangle.$$
Thus,
$$\sum_{I\in \mathscr{B}}(-1)^{|I|}\mu\left(\bigcap_{i\in I}A_i\right)=\sum_{i=1}^k\sum_{I\in \mathscr{B}_i}(-1)^{|I|}\mu\left(\bigcap_{i\in I}A_i\right)$$
$$=\sum_{i=1}^k\sum_{j=1}^t\sum_{I\in\langle I_j\rangle}(-1)^{|I|}\mu\left(\bigcap_{i\in I}A_i\right)=0.$$
So (6) follows directly, which completes our proof.
\end{proof}

\noindent{\bf Remark}. Theorem \ref{thm2} is an extension of Theorem \ref{thm1} and may reduce more terms:

Firstly, let $\mathfrak{X}$ be defined as in  Theorem \ref{thm1}. Set $\{B_1,B_2,\cdots,B_k\}=\mathfrak{X}$ and, for $i\in\{1,2,\cdots,k\}$,   set $B^*_i=B'_i$ and let $b_i=\min B'_i$ (the minimum element in $B'_i$). Without loss of generality, we may assume that $b_1\leq b_2\leq \cdots \leq b_k$.

If  $I\in \mathfrak{I}$, say $I$ contains exactly $B_{i_1},B_{i_2},\cdots,B_{i_p}$ with $p>0$ and $i_1<i_2<\cdots<i_p$, then we claim that $I\in\mathscr{B}_{i_1}$ and, therefore, $I\in\mathscr{B}$.

Suppose to the contrary that $I\notin\mathscr{B}_{i_1}$. Then there is $j<i_1$ such that $I\supseteq B_j\setminus B'_{i_1}$. On the other hand, by the minimality of $i_1$, we have $I\nsupseteq B_j$ since $j<i_1$. This means that there is $b\in B'_{i_1}$ such that $b\in B_j$. Therefore, $b<b_j$ since $b_j$ is an upper bound of $B_j$. This is a contradiction since $b_j\leq b_{i_1}\leq b$. Our claim follows.

Conversely, if $I\in\mathscr{B}$, say   $I\in\mathscr{B}_{i}$, then  we have $I\supseteq B_i$ and, therefore, $I\in \mathfrak{I}$.

As a result, we have $\mathfrak{I}=\mathscr{B}$. Thus, (\ref{main}) implies (3).

Secondly, if $\{B,B^*\}$ is a pair such that $B$ differs from $B_1,B_2,\cdots,B_k$; $\{B,B^*\}$ satisfies (5); $\{B,B'\}$ does not satisfy (2); $B_i\setminus B^*\nsubseteq B$ for any $i=1,2,\cdots,k$. Then $\mathscr{B}$ can contain $B$ as an element while $\mathfrak{X}$ and therefore $\mathfrak{I}$ cannot contain $B$ as an element. This means that $\mathscr{B}\supsetneqq \mathfrak{I}$, that is, (6) reduces more terms than (3) does. \hfill$\square$

\section{Examples in graph polynomials}
As examples, in this section we  apply Theorem \ref{thm2} to chromatic polynomial of hypergraph, independence and domination polynomial of graph. We will see that the ordering-free method reduces more terms than the ordering-based method.

Let  $P(G,x)$ be a graph polynomial of a graph $G$ represented in the form of inclusion-exclusion principle, i.e.,
\begin{equation*}
P(G,x)=\sum_{F\subseteq E(G)}(-1)^{|F|}p(F,x),
\end{equation*}
where $E(G)$ is the edge set of $G$ and $p(F,x)$ is a polynomial in $x$ associated with $F\subseteq E(G)$. We specialize the index set  $P$ to be $E(G)$ and, for any $F\subseteq E(G)$, set
\begin{equation}
\mu\left(\bigcap_{e\in F}{A_e}\right)=p(F,x).
\end{equation}

For a pair $B,B^*\subseteq E(G)$ with $B\cap B^*=\emptyset$, if $B^*$ is a single-edge set, say $B^*=\{b\}$, then the condition
\begin{equation*}
\bigcap_{e\in B}{A_e}\subseteq \bigcup_{e\in B^*}{A_e},
\end{equation*}
i.e., $\bigcap_{e\in B}{A_e}\subseteq A_b$, is equivalent to
\begin{equation*}
\bigcap_{e\in B}{A_e}=\bigcap_{e\in B\cup\{b\}}{A_e}.
\end{equation*}
Combining with (8), we have
\begin{equation}
p(B,x)=p(B\cup\{b\},x).
\end{equation}

Thus, a pair $\{B,\{b\}\}$ (viewed as $\{B_i,B_i^*\}$) satisfies the requirement of Theorem \ref{thm2} provided it satisfies (9). We refer to such pair $\{B,b\}$ as a {\it broken pair} of $P(G,x)$ and $B$ a {\it broken set} if $B$ is minimal (i.e., $B$ has no proper subset satisfying (9)). Further, given a linear ordering `<' on $E(G)$, we call $B$ a {\it broken pair with respect to} `<' if $\{b\}=B'$. By Theorem \ref{thm2} we have the following corollary immediately.
\begin{corollary}\label{cor} Let $\{B_1,B^*_1\},\{B_2,B^*_2\},\cdots,\{B_k,B^*_k\}$ be broken pairs of $P(G,x)$. Then
\begin{equation*}
P(G,x)=\sum_{F\in 2^{E(G)}\setminus \mathscr{B}}(-1)^{|F|}p(F,x).
\end{equation*}
\end{corollary}
\noindent{\bf Chromatic polynomial of hypergraph}. The chromatic polynomial $\chi(H,x)$ of a simple hypergraph $H$ counts the number of the vertex colorings such that each (hyper) edge of cardinality at least two has two vertices of distinct colors \cite{Berge,Dohmen01}. The following inclusion-exclusion expression was given in \cite{Dohmen01,Trinks}:
$$\chi(H,x)=\sum_{F\subseteq E(H)}(-1)^{|F|}x^{c(F)},$$
where $c(F)$ is the number of the components of the spanning subgraph of $H$ with edge set $F$.

Given a linear order `<' on the edge set $E(H)$,  Dohmen \cite{Dohmen01} generalized the Whitney's broken circuit theorem to hypergraph by extending the broken circuit defined on a cycle (see \cite{Berge} for the definition of a cycle), with a particular constraints that each edge of the cycle is included
by the union of the other edges of that cycle. A set $F\subseteq E(H)$ is called a {\it $\delta$-cycle} if $F$ is minimal such that $c(F\setminus\{f\})=c(F)$ for each $f\in F$. We note that every cycle with the above particular constraints is or contains a $\delta$-cycle while a $\delta$-cycle is not necessarily a cycle with this constraints. A set $B$ is called a {\it broken cycle} if $B$ is obtained from a $\delta$-cycle by deleting its maximum edge. In \cite{Trinks}, Trinks generalized the Dohmen's result by extending the broken circuit to broken cycle.

 For $B\subseteq E(H)$ and $b\in E(H)\setminus B$, by (9) it can be seen that $B$ is a broken set of $\chi(H,x)$ provided $B$ is minimal such that 
\begin{equation} 
c(B)=c(B\cup\{b\}).
\end{equation}
 We can see that the notion `broken set' for hypergraph is an extension of `broken cycle'. Moreover, in condition (10) there is no need to require $b$ to be the maximum edge of $B\cup\{b\}$ for a broken set.

Let's consider  the hypergraph $H=(V,E)$ with vertex set $V=\{1,2,3,4,5,6\}$ and edge set $E=\{\{1,2,3\},\{3,4,5\},\{2,3,4\},\{1,2,6\}\}$. We note that $H$ contains neither broken circuit (with the particular constraints) nor broken cycle,  no matter how to order its edges. This means that no terms in $\chi(H,x)$ can be reduced by broken circuit or broken cycle.

For an edge $\{i,j,k\}$ we write it simply as $ijk$. By (10) it can be seen that $H$ has two broken sets $B_1=\{123,345\}$ with $B_1^*=\{b_1\}=\{234\}$ and $B_2=\{234,126\}$ with $B_2^*=\{b_2\}=\{123\}$. Therefore,\\
$\mathscr{B}_1=\{\{123,345\},\{123,345,234\},\{123,345,126\},\{123,345,234,126\}\}$ and\\
$\mathscr{B}_2=\{\{234,126\},\{234,126,123\}\}$.

Consider the edge ordering  $123<345<234<126$. Again by (10), $H$ contains only one broken set with respect to `<', i.e., $B=\{123,345\}$ with $B'=\{234\}$. Thus, $\mathfrak{X}=\{B\}$ (see Theorem \ref{thm1}) and \\ $\mathfrak{I}=\{\{123,345\},\{123,345,234\},\{123,345,126\},\{123,345,234,126\}\}=\mathscr{B}_1$.

So by Theorem \ref{thm1} and Corollary \ref{cor},  the chromatic polynomial of $H$ is
$$\chi(H,x)=\sum_{F\in 2^{E}\setminus\mathfrak{I}}(-1)^{|F|}x^{c(F)}=\sum_{F\in 2^{E}\setminus(\mathscr{B}_1\cup\mathscr{B}_2)}(-1)^{|F|}x^{c(F)}=k^6-4k^4+3k^3+k^2-k.$$
Moreover, we see that $|2^{E}|=16>|2^{E}\setminus\mathfrak{I}|=12>|2^{E}\setminus(\mathscr{B}_1\cup\mathscr{B}_2)|=10$.

Finally, it can be seen that $H$ has at most one broken set with respect to `<', no matter how to  define the order `<'.

\noindent{\bf Independence polynomial of graph}. For  a graph $G$, the independence polynomial \cite{Gutman02,Hoede01} of $G$ can be represented as the following  inclusion-exclusion formula \cite{Dohmen05}:
\begin{equation}
I(G,x)=\sum_{F\subseteq E(G)}(-1)^{|F|}x^{|G[F]|}(1+x)^{n-|G[F]|},
\end{equation}
where $|G[F]|$ is the number of vertices in the subgraph of $G$ induced by $F$.

It was shown  \cite{Dohmen05} that the Whitney's broken circuit theorem is also valid for independence polynomial. By (9) and (11), a set $B$ of edges is a broken set provided $B$ is minimal such that $G[B]=G[B\cup\{b\}]$ for some $b\notin B$. This means that $B=\{e_1,e_2\}$ and $e_1be_2$ is a path or a cycle of length 3. We call such $B$ a {\it broken path}. We note that every broken circuit includes a broken path as a subgraph.

Let's consider  the path $G=e_1e_2e_3e_4$ of length 4 with edge ordering  $e_1<e_3<e_2<e_4$. Similar to the previous example, we have $B_1=\{e_1,e_3\}$ with $B_1^*=\{e_2\}$ and $B_2=\{e_2,e_4\}$ with $B_2^*=\{e_3\}$, and $\mathfrak{X}=\{\{e_1,e_3\}\}$. Therefore: \\
$\mathscr{B}_1=\{\{e_1,e_3\},\{e_1,e_2,e_3\},\{e_1,e_3,e_4\},\{e_1,e_2,e_3,e_4\}\}$;\\
$\mathscr{B}_2=\{\{e_2,e_4\},\{e_2,e_3,e_4\}\}$; and\\
$\mathfrak{I}=\{\{e_1,e_3\},\{e_1,e_2,e_3\},\{e_1,e_3,e_4\},\{e_1,e_2,e_3,e_4\}\}=\mathscr{B}_1.$

\noindent{\bf Domination polynomial of graph}. For a graph $G$ and $W\subseteq V(G)$, denote by $N[W]$ the closed neighbourhood of $W$, i.e.,
$$N[W]=W\cup\{v:v\ {\rm is\ adjacent\ to\ some\ vertex\ in}\ W\}.$$

Let $d_i$ be the number of the sets $W$ of $i$ vertices such that $N_G[W]=V(G)$. The domination polynomial $D(G,x)$ is defined by
$D(G,x)=\sum_{i=1}^nd_ix^i.$ The following form was given in \cite{Dohmen08},
\begin{equation}
D(G,x)=\sum_{W\subseteq V(G)}(-1)^{|W|}(1+x)^{n-|N[W]|}.
\end{equation}

 A set $B$ is called {\it broken neighbourhood} if $B=N(v)$ and $v=\max N[v]$. In \cite{Dohmen08}, Dohmen and Tittmann proved that the sum in (12) can be restricted to those subsets of vertices which do not contain any  broken neighbourhood.

Due to (12), we replace the role of edges in (9) by vertices. For $B\subseteq V(G)$ and $b\in V(G)\setminus B$, by (9) it can be seen that $B$ is a broken set of $D(G,x)$ provided $B$ is minimal such that $|N[B]|=|N[B\cup\{b\}]|,$ i.e.,
\begin{equation}
N[b]\subseteq N[B].
\end{equation}
We can see that the`broken set' of $D(G,x)$ is an extension of `broken neighbourhood'.

Consider  the path $P=v_1v_2v_3v_4$ with vertex ordering $v_1<v_4<v_3<v_2$. Similarly, by (13) we have $B_1=\{v_1,v_3\}$ with $B_1^*=\{v_2\}$, $B_2=\{v_1,v_4\}$ with $B_2^*=\{v_2\}$  and $B_3=\{v_2,v_4\}$ with $B_3^*=\{v_3\}$, and $\mathfrak{X}=\{\{v_1,v_3\},\{v_1,v_4\}\}$. Therefore: \\
$\mathscr{B}_1=\{\{v_1,v_3\},\{v_1,v_2,v_3\},\{v_1,v_3,v_4\},\{v_1,v_2,v_3,v_4\}\}$;\\
$\mathscr{B}_2=\{\{v_1,v_4\},\{v_1,v_2,v_4\}\}$;\\
$\mathscr{B}_3=\{\{v_2,v_4\},\{v_2,v_3,v_4\}\}$; and\\
$\mathfrak{I}=\{\{v_1,v_3\},\{v_1,v_2,v_3\},\{v_1,v_3,v_4\},\{v_1,v_2,v_3,v_4\},\{v_1,v_4\},\{v_1,v_2,v_4\}\}=\mathscr{B}_1\cup\mathscr{B}_2.$

\section*{Acknowledgments}
This work was supported by the National Natural Science Foundation of China [Grant numbers, 11471273, 11561058].


\begin{thebibliography}{99}
\bibitem{Berge} C. Berge, Hypergraphs. Vol. 45. North-Holland Mathematical Library. North-Holland, 1989.
\bibitem{Dohmen02} K. Dohmen, An improvement of the inclusion-exclusion principle, Arch. Math.  72(4)(1999) 298-303.


\bibitem{Dohmen05} K. Dohmen, A. P\"onitz, P. Tittmann, A new two-variable generalization of the chromatic polynomial, Discrete Math. Theor. Comput Sci.  6(1)(2003) 69-90.

\bibitem{Dohmen08} K. Dohmen, P. Tittmann, Domination Reliability, Electron. J. Combin. 19(1)(2012) $\#$P15.

\bibitem{Dohmen01} K. Dohmen, M. Trinks, An Abstraction of Whitney's Broken Circuit Theorem, Electron. J. Combin. 21(4)(2014)  $\#$P4. 32.

\bibitem{Gutman02} I. Gutman, F. Harary, Generalizations of the matching polynomial, Utilitas Math. 24(1)(1983) 97-106.

\bibitem{Hoede01} C. Hoede, X. Li, Clique polynomials and independent set polynomials of graphs, Discrete Math. 125(1-3)(1994) 219-228.

\bibitem{Liao01} Y. Liao, Y. Hou, Note on the Subgraph Component Polynomial, Electron. J. Combin. 21(3)(2014)  $\#$P3. 27.

\bibitem{Narushima01} H. Narushima, Principle of inclusion-exclusion on semilattices, J. Combin. Theory, Ser. A 17(1974) 196-203.

\bibitem{Riordan01} J. Riordan, An introduction to combinatorial analysis, John Wiley $\&$ Sons, Inc., New York, 1958.

\bibitem{Rota01} G.C. Rota, On the foundations of combinatorial theory I. Theory of M\"obius functions, Prob. Theory Related Fields 2(4)(1964) 340-368.

\bibitem{Trinks} M. Trinks, A note on a Broken-cycle Theorem for hypergraphs, Discuss. Math. Graph Theory  34 (3)(2014) 641-646

\bibitem{Whitney01} H. Whitney, A logical expansion in mathematics, Bull. Amer. Math. Soc. 38(8)(1932) 572-579.

\end{thebibliography}
\end{document}